\documentclass{amsart}
\usepackage{amsmath, amsthm, amssymb, amscd, mathrsfs, eucal, epsfig}
\usepackage{pinlabel}

\begin{document}

\newtheorem{theorem}{Theorem}[section]
\newtheorem{lemma}[theorem]{Lemma}
\newtheorem{sublemma}[theorem]{Sublemma}
\newtheorem{proposition}[theorem]{Proposition}
\newtheorem{corollary}[theorem]{Corollary}
\newtheorem{conjecture}[theorem]{Conjecture}
\newtheorem{question}[theorem]{Question}
\newtheorem{problem}[theorem]{Problem}
\newtheorem*{claim}{Claim}
\newtheorem*{criterion}{Criterion}
\newtheorem*{simple_thm}{Theorem A}
\newtheorem*{ennd_thm}{Theorem B}

\theoremstyle{definition}
\newtheorem{definition}[theorem]{Definition}
\newtheorem{construction}[theorem]{Construction}
\newtheorem{notation}[theorem]{Notation}

\theoremstyle{remark}
\newtheorem{remark}[theorem]{Remark}
\newtheorem{example}[theorem]{Example}

\numberwithin{equation}{subsection}

\def\Z{\mathbb Z}
\def\N{\mathbb N}
\def\R{\mathbb R}
\def\Q{\mathbb Q}
\def\D{\mathcal D}
\def\E{\mathcal E}
\def\RR{\mathcal R}
\def\P{\mathcal P}
\def\F{\mathcal F}

\def\cl{\textnormal{cl}}
\def\scl{\textnormal{scl}}
\def\Aut{\textnormal{Aut}}
\def\homeo{\textnormal{Homeo}}
\def\rot{\textnormal{rot}}
\def\area{\textnormal{area}}
\def\asdim{\textnormal{as dim}}
\def\supp{\textnormal{supp}}

\def\Id{\textnormal{Id}}
\def\id{\textnormal{id}}
\def\SL{\textnormal{SL}}
\def\PSL{\textnormal{PSL}}
\def\length{\textnormal{length}}
\def\fill{\textnormal{fill}}
\def\rank{\textnormal{rank}}
\def\til{\widetilde}

\title{Large scale geometry of commutator subgroups}
\author{Danny Calegari}
\address{Department of Mathematics \\ Caltech \\
Pasadena CA, 91125}
\email{dannyc@its.caltech.edu}
\author{Dongping Zhuang}
\address{Department of Mathematics \\ Caltech \\
Pasadena CA, 91125}
\email{dongping@its.caltech.edu}

\date{9/30/2008, Version 0.05}

\begin{abstract}
Let $G$ be a finitely presented group, and $G'$ its commutator subgroup. Let $C$ be the Cayley
graph of $G'$ with {\em all commutators in $G$} as generators. Then $C$ is
large scale simply connected. Furthermore, if $G$ is a torsion-free nonelementary
word-hyperbolic group, $C$ is one-ended. Hence (in this case), the asymptotic dimension of
$C$ is at least $2$.
\end{abstract}

\maketitle

\section{Introduction}

Let $G$ be a group and let $G':=[G,G]$ denote the commutator subgroup of $G$.
The group $G'$ has a canonical generating set $S$, which consists precisely
of the set of commutators of pairs of elements in $G$.
In other words,
$$S = \lbrace [g,h] \text{ such that } g,h \in G \rbrace$$
Let $C_S(G')$ denote the Cayley graph of $G'$ with respect to the generating set $S$.
This graph can be given the structure of a (path) metric space in the usual way,
where edges have length $1$ by fiat.

By now it is standard to expect that the large scale geometry of a Cayley graph will
reveal useful information about a group. However, one usually studies finitely generated
groups $G$ and the geometry of a Cayley graph $C_T(G)$ associated to a finite generating set $T$.
For typical infinite groups $G$, the set of commutators $S$ will be infinite, and the
Cayley graph $C_S(G')$ will not be locally compact. This is a significant complication.
Nevertheless, $C_S(G')$ has several distinctive properties which invite careful study:

\begin{enumerate}
\item{The set of commutators of a group is {\em characteristic} (i.e. invariant under any
automorphism of $G$), and therefore 
the semi-direct product $G' \rtimes \Aut(G)$ acts on $C_S(G')$ by isometries}
\item{The metric on $G'$ inherited as a subspace of $C_S(G')$ is both left- and
right-invariant (unlike the typical Cayley graph, whose metric is merely left-invariant)}
\item{Bounded cohomology in $G$ is reflected in the geometry of $G'$; for instance, 
the translation length $\tau(g)$ of an element $g \in G'$ is the stable commutator
length $\scl(g)$ of $g$ in $G$}
\item{Simplicial loops in $C_S(G')$ through the origin 
correspond to (marked) homotopy classes of maps of closed surfaces to a $K(G,1)$}
\end{enumerate}

These properties are straightforward to establish; for details, see \S~\ref{definition_section}. 

This paper concerns the connectivity of $C_S(G')$ in the large for various groups $G$. 
Recall that a {\em thickening} $Y$ of a
metric space $X$ is an isometric inclusion $X \to Y$ into a bigger metric space, such that
the Hausdorff distance in $Y$ between $X$ and $Y$ is finite. A metric space $X$
is said to be {\em large scale $k$-connected} if for any thickening $Y$ of $X$ there is another
thickening $Z$ of $Y$ which is $k$-connected (i.e. $\pi_i(Z) = 0$ for $i \le k$; also see
the Definitions in \S~\ref{simple_section}). Our first
main theorem, proved in \S~\ref{simple_section}, 
concerns the large scale connectivity of $C_S(G')$ where $G$ is finitely presented:

\begin{simple_thm}
Let $G$ be a finitely presented group. Then $C_S(G')$ is large scale simply connected.
\end{simple_thm}

As well as large scale connectivity, one can study connectivity {\em at infinity}.
In \S~\ref{hyperbolic_group_section} we specialize to word-hyperbolic $G$ and
prove our second main theorem, concerning the connectivity of $G'$ at infinity:

\begin{ennd_thm}
Let $G$ be a torsion-free nonelementary word-hyperbolic group. Then $C_S(G')$ is one-ended; 
i.e. for any $r>0$ there is an $R\ge r$ such that any two points in $C_S(G')$ at distance
at least $R$ from $\id$ can be joined by a path which does not come closer than distance $r$
to $\id$.
\end{ennd_thm}

Combined with a theorem of Fujiwara-Whyte \cite{Fujiwara_Whyte}, Theorem~A and Theorem~B
together imply that for $G$ a torsion-free nonelementary word-hyperbolic group, $C_S(G')$ 
has asymptotic dimension at least $2$ 
(see \S~\ref{asymptotic_section} for the definition of asymptotic dimension).

\section{Definitions and basic properties}\label{definition_section}

Throughout the rest of this paper, $G$ will denote a group, $G'$ will denote its commutator subgroup, and
$S$ will denote the set of (nonzero) commutators in $G$, thought of as a generating set for $G'$.
Let $C_S(G')$ denote the Cayley graph of $G'$ with respect to the generating set $S$. As a graph,
$C_S(G')$ has one vertex for every element of $G'$, and two elements $g,h \in G'$ are joined by an
edge if and only if $g^{-1}h \in S$. Let $d$ denote distance in $C_S(G')$ restricted to $G'$.

\begin{definition}
Let $g \in G'$. The {\em commutator length} of $g$, denoted $\cl(g)$, is the smallest number of
commutators in $G$ whose product is equal to $g$.
\end{definition}

From the definition, it follows that $\cl(g) = d(\id,g)$ and $d(g,h) = \cl(g^{-1}h)$ for
$g,h \in G'$.

\begin{lemma}
The group $G' \rtimes \Aut(G)$ acts on $C_S(G')$ by isometries.
\end{lemma}
\begin{proof}
$\Aut(G)$ acts as permutations of $S$, and therefore the natural action on $G$ extends
to $C_S(G')$. Further, $G'$ acts on $C_S(G')$ by left multiplication.
\end{proof}

\begin{lemma}\label{left_right_invariant}
The metric on $C_S(G')$ restricted to $G'$ is left- and right-invariant.
\end{lemma}
\begin{proof}
Since the inverse of a commutator is a commutator,
we have $\cl(g^{-1}h) = \cl(h^{-1}g)$. Since the conjugate of a commutator by any element is a commutator,
we have $\cl(h^{-1}g) = \cl(gh^{-1})$. This completes the proof.
\end{proof}

\begin{definition}
Given a metric space $X$ and an isometry $h$ of $X$, the {\em translation length}
of $h$ on $X$, denoted $\tau(h)$, is defined by the formula
$$\tau(h) = \lim_{n \to \infty} \frac {d(p,h^n(p))} n$$
where $p \in X$ is arbitrary. 
\end{definition}

By the triangle inequality, the limit does not depend on the choice of $p$.

For $g \in G'$ acting on $C_S(G')$ by left multiplication, we can take $p = \id$. Then
$d(\id,g^n(\id)) = \cl(g^n)$.

\begin{definition}
Let $G$ be a group, and $g \in G'$. The {\em stable commutator length} of $g$ is the limit
$$\scl(g) = \lim_{n \to \infty} \frac {\cl(g^n)} n$$
\end{definition}

Hence we have the following:
\begin{lemma}
Let $g \in G'$ act on $C_S(G')$ by left multiplication. There is an equality $\tau(g) = \scl(g)$.
\end{lemma}
\begin{proof}
This is immediate from the definitions.
\end{proof}

Stable commutator length is related to two-dimensional (bounded) cohomology. For
an introduction to stable commutator length, see \cite{Calegari_scl}; for an
introduction to bounded cohomology, see \cite{Gromov_bounded}.

If $X$ is a metric space, and $g$ is an isometry of $X$, one can obtain lower bounds
on $\tau(g)$ by constructing a Lipschitz function on $X$ which grows linearly on the
orbit of a point under powers of $g$. One important class of Lipschitz functions
on $C_S(G')$ are {\em quasimorphisms}:

\begin{definition}
Let $G$ be a group. A function $\phi:G \to \R$ is a {\em quasimorphism} if there is a
least positive real number $D(\phi)$ called the {\em defect}, such that for all
$g, h \in G$ there is an inequality
$$|\phi(g) + \phi(h) - \phi(gh)| \le D(\phi)$$
\end{definition}

From the defining property of a quasimorphism, $|\phi(\id)| \le D(\phi)$ and
therefore by repeated application of the triangle inequality, one can estimate
$$|\phi(f[g,h]) - \phi(f)| \le 7D(\phi)$$ 
for any $f,g,h \in G$. In other words,
\begin{lemma}\label{7_lipschitz}
Let $G$ be a group, and let $\phi:G \to \R$ be a quasimorphism with defect
$D(\phi)$. Then $\phi$ restricted to $G'$ is $7D(\phi)$-Lipschitz in the metric
inherited from $C_S(G')$.
\end{lemma} 

Word-hyperbolic groups admit a rich family of quasimorphisms. We will exploit this
fact in \S~\ref{hyperbolic_group_section}.

\section{Large scale simple connectivity}\label{simple_section}

The following definitions are taken from \cite{Gromov_asymptotic}, pp.~23--24.
\begin{definition}
A {\em thickening} $Y$ of a metric space $X$ is an isometric inclusion $X \to Y$
with the property that there is a constant $C$ so that every point in $Y$ is within
distance $C$ of some point in $X$.
\end{definition}

\begin{definition}
A metric space $X$ is {\em large scale $k$-connected} if for every thickening
$X \subset Y$ there is a thickening $Y \subset Z$ which is $k$-connected in the
usual sense; i.e. $Z$ is path-connected, and $\pi_i(Z) = 0$ for $i \le k$.
\end{definition}

For $G$ a finitely generated group with generating set $T$, Gromov outlines a proof
(\cite{Gromov_asymptotic}, 1.$C_2$) that the Cayley
graph $C_T(G)$ is large scale $1$-connected if and only if $G$ is finitely presented,
and $C_T(G)$ is large scale $k$-connected if and only if there exists a proper
simplicial action of $G$ on a $(k+1)$-dimensional $k$-connected simplicial complex
$X$ with compact quotient $X/G$.

For $T$ an infinite generating set, large scale simple connectivity is equivalent to
the assertion that $G$ admits a presentation $G = \langle T \; | \; R \rangle$ where
all elements in $R$ have {\em uniformly bounded length} as words in $T$; i.e.
all relations in $G$ are consequences of relations of bounded length.

To show that $C_S(G')$ is large scale $1$-connected, it suffices to show that there
is a constant $K$ so that for
every simplicial loop $\gamma$ in $C_S(G')$ there are a sequence of loops
$\gamma = \gamma_0,\gamma_1,\cdots,\gamma_n$ where $\gamma_n$ is the trivial loop,
and each $\gamma_i$ is obtained from
$\gamma_{i-1}$ by cutting out a subpath $\sigma_{i-1} \subset \gamma_{i-1}$ 
and replacing it by a subpath
$\sigma_i \subset \gamma_i$ 
with the same endpoints, so that $|\sigma_{i-1}| + |\sigma_i| \le K$.

More generally, we call the operation of cutting out a subpath $\sigma$ and replacing
it by a subpath $\sigma'$ with the same endpoints where $|\sigma| + |\sigma'| \le K$ a
{\em $K$-move}.

\begin{definition}
Two loops $\gamma$ and $\gamma'$ are {\em $K$-equivalent} if there is a finite sequence
of $K$-moves which begins at $\gamma$, and ends at $\gamma'$.
\end{definition}

$K$-equivalence is (as the name suggests) an equivalence relation.
The statement that $C_S(G')$ is large scale $1$-connected is equivalent to the statement
that there is a constant $K$ such that every two loops in $C_S(G')$ are $K$-equivalent.

First we establish large scale simple connectivity in the case of a free group.

\begin{lemma}\label{free_group_case}
Let $F$ be a finitely generated free group. Then $C_S(F')$ is large scale
simply connected.
\end{lemma}
\begin{proof}
Let $\gamma$ be a loop in $C_S(F')$. After acting on $\gamma$ by left translation, we
may assume that $\gamma$ passes through $\id$, so we may think of $\gamma$ as a simplicial
path in $C_S(F')$ which starts and ends at $\id$. If $s_i \in S$ corresponds to the
$i$th segment of $\gamma$, we obtain an expression
$$s_1 s_2 \cdots s_n = \id$$
in $F$, where each $s_i$ is a commutator. For each $i$, let $a_i,b_i \in F$ be elements
with $[a_i,b_i] = s_i$ (note that $a_i,b_i$ with this property are not necessarily unique).
Let $\Sigma$ be a surface of genus $n$, and let $\alpha_i,\beta_i$ for $i \le n$ be a
standard basis for $\pi_1(\Sigma)$; see Figure~\ref{standard_marking}. 

\begin{figure}[htpb]
\labellist
\small\hair 2pt
\endlabellist
\centering
\includegraphics[scale=0.4]{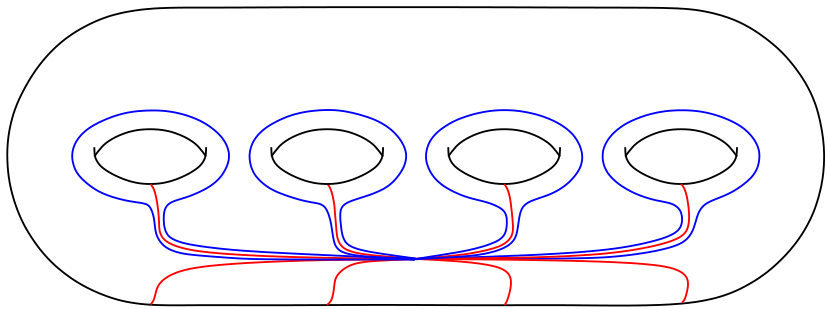}
\caption{A standard basis for $\pi_1(\Sigma)$ where $\Sigma$ has genus $4$.
The $\alpha_i$ curves are in red, and the $\beta_i$ curves are in blue.} \label{standard_marking}
\end{figure}

Let $X$ be a wedge of circles corresponding to
free generators for $F$, so that $\pi_1(X) = F$. We can construct a basepoint preserving
map $f:\Sigma \to X$ with $f_*(\alpha_i) = a_i$ and $f_*(\beta_i) = b_i$ for each $i$.
Since $X$ is a $K(F,1)$, the homotopy class of $f$ is uniquely determined by the
$a_i,b_i$. Informally, we could say that loops in $C_S(F')$ correspond to based homotopy classes
of maps of marked oriented surfaces into $X$ (up to the ambiguity indicated above).

Let $\phi$ be a (basepoint preserving) self-homeomorphism of $\Sigma$. The map 
$f \circ \phi:\Sigma \to X$ determines a new loop in $C_S(F')$ (also passing through
$\id$) which we denote $\phi_*(\gamma)$ (despite the notation, this image does not depend
only on $\gamma$, but on the choice of elements $a_i,b_i$ as above).

\begin{sublemma}\label{automorphism_bounded}
There is a universal constant $K$ independent of $\gamma$ or of $\phi$ (or even of $F$)
so that after composing $\phi$ by an inner automorphism of $\pi_1(\Sigma)$ if necessary,
$\gamma$ and $\phi_*(\gamma)$ as above are $K$-equivalent.
\end{sublemma}
\begin{proof}
Suppose we can express $\phi$ as a product of (basepoint preserving) automorphisms 
$$\phi = \phi_m \circ \phi_{m-1} \circ \cdots \circ \phi_1$$ 
such that if $\alpha_i^j,\beta_i^j$ denote the images of $\alpha_i,\beta_i$
under $\phi_j \circ \phi_{j-1} \circ \cdots \circ \phi_1$, then $\phi_{j+1}$ fixes
all but $K$ consecutive pairs $\alpha_i^j,\beta_i^j$ up to (basepoint preserving) homotopy. Let
$s_i^j = [f_*\alpha_i^j,f_*\beta_i^j]$, and let $\gamma^j$ be the loop in $C_S(F')$
corresponding to the identity $s_1^js_2^j \cdots s_n^j = \id$ in $F$.

For each $j$, let $\supp_{j+1}$ denote the {\em support} of $\phi_{j+1}$; i.e. the set of
indices $i$ such that $\phi_{j+1}(\alpha_i^j) \ne \alpha_i^j$ or $\phi_{j+1}(\beta_i^j) \ne \beta_i^j$.
By hypothesis, $\supp_{j+1}$ consists of at most $K$ indices for each $j$.

Because it is just the marking on $\Sigma$ which has been changed and not the map $f$,
if $k \le i \le k+K-1$ is a maximal consecutive string of indices in $\supp_{j+1}$, 
then there is an equality of products
$$s_k^j s_{k+1}^j \cdots s_{k+K-1}^j = s_k^{j+1} s_{k+1}^{j+1} \cdots s_{k+K-1}^{j+1}$$
as elements of $F$. This can be seen geometrically as follows. The expression on the
left is the image under $f_*$ of an element represented by a certain embedded based loop
in $\Sigma$, while the expression on the right is its image under $f_* \circ \phi_{j+1}$.
The automorphism $\phi_{j+1}$ is represented by a homeomorphism of $\Sigma$
whose support is contained in regions bounded by such loops. Hence the expressions are equal.
It follows that $\gamma^j$ and $\gamma^{j+1}$ are $2K$-equivalent.

So to prove the Sublemma it suffices to show that any automorphism of $S$ can be expressed
(up to inner automorphism) as a product of automorphisms $\phi_i$ with the property above. 

The hypothesis that we may compose $\phi$ by an inner automorphism means that we need
only consider the image of $\phi$ in the mapping class group of $\Sigma$.
It is well-known since Dehn \cite{Dehn_reference} that the 
mapping class group of a closed oriented
surface $\Sigma$ of genus $g$ is generated by twists in a finite standard set
of curves, each of which intersects at most two of the $\alpha_i,\beta_i$
essentially; see Figure~\ref{standard_twists}.

\begin{figure}[htpb]
\labellist
\small\hair 2pt
\endlabellist
\centering
\includegraphics[scale=0.4]{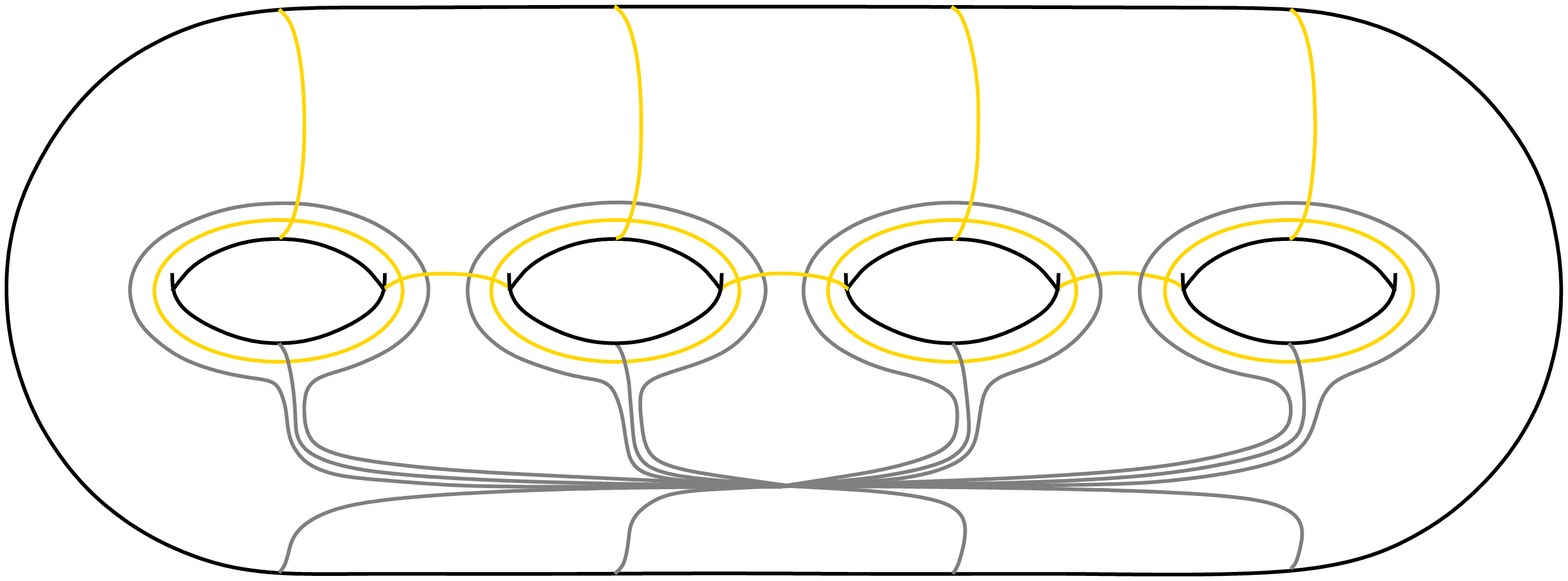}
\caption{A standard set of $3g-1$ simple curves, in yellow. Dehn twists in these curves generate the
mapping class group of $\Sigma$.} \label{standard_twists}
\end{figure}

So write
$\phi = \tau_1 \tau_2 \cdots \tau_m$ where the $\tau_i$ are all standard generators. 
Now define
$$\phi_j = \tau_1 \tau_2 \cdots \tau_{j-1} \tau_j \tau_{j-1}^{-1} \cdots \tau_1^{-1}$$
We have 
$$\phi_j \phi_{j-1} \cdots \phi_1 = \tau_1 \tau_2 \cdots \tau_j$$
Moreover, each $\phi_j$ is a Dehn twist in a curve which is the image of a standard curve under
$\phi_{j-1} \cdots \phi_1$, and therefore intersects 
$\alpha_i^{j-1},\beta_i^{j-1}$ essentially for at most $2$ (consecutive) indices $i$. 
This completes the proof of the Sublemma (and shows, in fact, that we can take $K=4$).
\end{proof}

We now complete the proof of the Lemma. As observed by Stallings (see e.g. \cite{Stallings}), 
a nontrivial map $f:\Sigma \to X$ from a closed, oriented
surface to a wedge of circles factors (up to homotopy) through a {\em pinch} in the following
sense. Make $f$ transverse to some edge $e$ of $X$, and look at the preimage
$\Gamma$ of a regular value of $f$ in $e$. After homotoping inessential loops of $\Gamma$ off $e$,
we may assume that for some edge $e$ and some regular value, the preimage $\Gamma$ contains
an embedded essential loop $\delta$. 

There are two cases to consider. In the first case,
$\delta$ is nonseparating. In this case, let $\phi$ be an automorphism which takes $\alpha_1$ to
the free homotopy class of $\delta$. 
Then $\gamma$ and $\phi_*(\gamma)$ are $K$-equivalent by the Sublemma. However,
since $f(\delta)$ is homotopically trivial in $X$, there is an identity $[\phi_*\alpha_1,\phi_*\beta_1] = \id$
and therefore $\phi_*(\gamma)$ has length $1$ shorter than $\gamma$.

In the second case, $\phi$ is separating, and we can let $\phi$ be an automorphism which takes
the free homotopy class of $[\alpha_1,\beta_1]\cdots[\alpha_j,\beta_j]$ to $\delta$. Again,
by the Sublemma, $\gamma$ and $\phi_*(\gamma)$ are $K$-equivalent. But now $\phi_*(\gamma)$
contains a subarc of length $j$ with both endpoints at $\id$, so we may write it as a product
of two loops at $\id$, each of length shorter than that of $\gamma$.

By induction, $\gamma$ is $K$-equivalent to the trivial loop, and we are done.
\end{proof}

We are now in a position to prove our first main theorem.

\begin{simple_thm}
Let $G$ be a finitely presented group. Then $C_S(G')$ is large scale simply connected.
\end{simple_thm}
\begin{proof}
Let $W$ be a smooth $4$-manifold (with boundary) satisfying $\pi_1(W) = G$. If
$G = \langle T \; | \; R \rangle$ is a finite presentation, we can build $W$ as
a handlebody, with one $0$-handle, one $1$-handle for every generator in $T$, and
one $2$-handle for every relation in $R$. If $r_i \in R$ is a relation, let $D_i$ be
the cocore of the corresponding $2$-handle, so that $D_i$ is a properly embedded
disk in $W$. Let $V \subset W$ be the union of the $0$-handle and the $1$-handles.
Topologically, $V$ is homotopy equivalent to a wedge of circles.
By the definition of cocores, the complement of $\cup_i D_i$ in $W$
deformation retracts to $V$. See e.g. \cite{Kirby_4_manifolds}, Chapter~1 for an introduction 
to handle decompositions of $4$-manifolds.

Given $\gamma$ a loop in $C_S(G')$, translate it by left multiplication so that it passes
through $\id$. As before, let $\Sigma$ be a closed oriented marked surface, and
$f:\Sigma \to W$ a map representing $\gamma$.

Since $G$ is finitely presented, $H_2(G;\Z)$ is finitely generated. Choose finitely many
closed oriented surfaces $S_1,\cdots, S_r$ in $W$ which generate $H_2(G;\Z)$. 
Let $K'$ be the supremum of the genus of the $S_i$.
We can choose a basepoint on each $S_i$, and maps to $W$ which are basepoint preserving.
By tubing $\Sigma$ repeatedly to copies of the $S_i$ with either orientation, we obtain
a new surface and map $f':\Sigma' \to W$ representing a loop $\gamma'$ such that
$f'(\Sigma')$ is null-homologous in $W$, and $\gamma'$ is $K'$-equivalent to $\gamma$
(note that $K'$ depends on $G$ but not on $\gamma$). 

Put $f'$ in general position
with respect to the $D_i$ by a homotopy. Since $f'(\Sigma')$ is null-homologous,
for each proper disk $D_i$, the signed intersection number vanishes:
$D_i \cap f'(\Sigma') = 0$.
Hence $f'(\Sigma)\cap D_i = P_i$ is a finite, even number of points which can be partitioned into
two sets of equal size corresponding to the local intersection number of $f'(\Sigma')$
with $D_i$ at $p \in P_i$. 

Let $p,q \in P_i$ have opposite signs, and let $\mu$ be an
embedded path in $D_i$ from $f'(p)$ to $f'(q)$. Identifying $p$ and $q$ implicitly with their
preimages in $\Sigma'$, let $\alpha$ and $\beta$ be arcs in
$\Sigma'$ from the basepoint to $(f')^{-1}p$ and $(f')^{-1}q$. 
Since $\mu$ is contractible, there is a neighborhood of $\mu$ in $D_i$ on which
the normal bundle is trivializable. Hence, since
$f'(\Sigma')$ and $D_i$ are transverse, we can find a neighborhood $U$ of $\mu$ in $W$
disjoint from the other $D_j$, and co-ordinates on $U$ satisfying
\begin{enumerate}
\item{$D_i \cap U$ is the plane $(x,y,0,0)$}
\item{$\mu \cap U$ is the interval $(t,0,0,0)$ for $t \in [0,1]$}
\item{$f'(\Sigma')\cap U$ is the union of the planes $(0,0,z,w)$ and $(1,0,z,w)$}
\end{enumerate}
Let $A$ be the annulus consisting of points $(t,0,\cos(\theta),\sin(\theta))$ where
$t \in [0,1]$. Then $A$ is disjoint from $D_i$ and all the other $D_j$, and we
can tube $f'(\Sigma')$ with $A$ to reduce the number of intersection points of $f'(\Sigma')$
with $\cup_i D_i$, at the cost of raising the genus by $1$. Technically, we remove the
disks $(f')^{-1}(0,0,s\cos(\theta),s\sin(\theta))$ and $(f')^{-1}(1,0,s\cos(\theta),s\sin(\theta))$
for $s \in [0,1]$ from $\Sigma'$, and sew in a new annulus which we map homeomorphically to $A$.
The result is $f'':\Sigma'' \to W$ with two fewer
intersection points with $\cup_i D_i$. 
This has the effect of adding a new (trivial) edge to the start of $\gamma'$, which is the commutator
of the elements represented by the core of $A$ and the loop $f'(\alpha)*\mu*f'(\beta)$.
Let $\gamma''$ denote this resulting loop, and observe that $\gamma''$ is $1$-equivalent to $\gamma'$.
After finitely many operations of this kind, we obtain $f''':\Sigma''' \to W$ corresponding to
a loop $\gamma'''$ which is $\max(1,K')$-equivalent to $\gamma$, such that
$f'''(\Sigma''')$ is disjoint from $\cup_i D_i$.

After composing with a deformation retraction, we may assume $f'''$ maps $\Sigma'''$ into $V$.
Let $F = \pi_1(V)$, and let $\rho:F \to G$ be the homomorphism induced by the inclusion
$V \to W$. There is a loop $\gamma^F$ in $C_S(F')$ corresponding to $f'''$ such that
$\rho_*(\gamma^F) = \gamma'''$ under the obvious simplicial map $\rho_*:C_S(F') \to C_S(G')$.
By Lemma~\ref{free_group_case}, the loop 
$\gamma^F$ is $K$-equivalent to a trivial loop in $C_S(F')$. Pushing forward the sequence
of intermediate loops by $\rho_*$ shows that $\gamma'''$ is $K$-equivalent to a trivial loop
in $C_S(G')$. Since $\gamma$ was arbitrary, we are done.
\end{proof}

\begin{remark}
A similar, though perhaps more combinatorial argument could be made working
directly with $2$-complexes in place of $4$-manifolds.
\end{remark}

In words, Theorem~A says that for $G$ a finitely presented group, all relations
amongst the commutators of $G$ are consequences of relations involving only
boundedly many commutators.

The next example shows that the size of this bound depends on $G$:

\begin{example}
Let $\Sigma$ be a closed surface of genus $g$, and $G = \pi_1(\Sigma)$.
If $\gamma$ is a loop in $C_S(G)$ through the origin, and $f:\Sigma' \to \Sigma$
is a corresponding map of a closed surface, then the homology class of $\Sigma'$
is trivial unless the genus of $\Sigma'$ is at least as big as that of $\Sigma$.
Hence the loop in $C_S(G)$ of length $g$ corresponding to the relation
in the ``standard'' presentation of $\pi_1(\Sigma)$ is not $K$-equivalent to the
trivial loop whenever $K < g$.
\end{example}

In light of Theorem~A, it is natural to ask the following question:

\begin{question}
Let $G$ be a finitely presented group. Is $C_S(G')$ large scale $k$-connected
for all $k$?
\end{question}

\begin{remark}
Laurent Bartholdi has pointed out that for $F$ a finitely generated free group,
there is a confluent, Noetherian rewriting system for $F'$, with rules of bounded
length, which puts every word in $F'$ over generators $S$ into normal form (with
respect to a ``standard'' free generating set for $F'$). By results
of Groves (\cite{Groves}) this should imply that $C_S(F')$ is large scale $k$-connected for all $k$,
but we have not verified this implication carefully. In any case, it gives
another more algebraic proof of Lemma~\ref{free_group_case}.
\end{remark}

\section{Word-hyperbolic groups}\label{hyperbolic_group_section}

In this section we specialize to the class of {\em word-hyperbolic groups}.
See \cite{Gromov_hyperbolic} for more details.

\begin{definition}
A path metric space $X$ is {\em $\delta$-hyperbolic} for some $\delta \ge 0$ if for
every geodesic triangle $abc$, and every point $p$ on the edge $ab$, there is
$q \in ac \cup bc$ with $d_X(p,q) \le \delta$. In other words, the $\delta$ neighborhood of
any two sides of a geodesic triangle contains the third side.
\end{definition}

\begin{definition}
A group $G$ is {\em word-hyperbolic} if there is a finite generating set $T$ for $G$ such that
$C_T(G)$ is $\delta$-hyperbolic as a path metric space, for some $\delta$.
\end{definition}

\begin{example}
Finitely generated free groups are word-hyperbolic. The fundamental group of a closed surface
with negative Euler characteristic is word-hyperbolic. Discrete cocompact groups of isometries
of hyperbolic $n$-space are word-hyperbolic.
\end{example}

To rule out some trivial examples, one makes the following:

\begin{definition}
A word-hyperbolic group is {\em elementary} if it has a cyclic subgroup of finite index,
and {\em nonelementary} otherwise.
\end{definition}

The main theorem we prove in this section concerns the geometry of $C_S(G')$ at infinity,
where $G$ is a nonelementary word-hyperbolic group. For the sake of brevity we restrict attention
to torsion-free $G$, though this restriction is not logically necessary; see Remark~\ref{torsion_remark}.

\begin{ennd_thm}
Let $G$ be a torsion-free nonelementary word-hyperbolic group. Then $C_S(G')$ is one-ended; 
i.e. for any $r>0$ there is an $R\ge r$ such that any two points in $C_S(G')$ at distance
at least $R$ from $\id$ can be joined by a path which does not come closer than distance $r$
to $\id$.
\end{ennd_thm}

We will estimate distance to $\id$ in $C_S(G')$ using quasimorphisms, as indicated in
\S~\ref{definition_section}. Hyperbolic groups admit a rich family of quasimorphisms.
Of particular interest to us are the {\em Epstein-Fujiwara counting quasimorphisms},
introduced in \cite{Epstein_Fujiwara}, generalizing a construction due to Brooks \cite{Brooks}
for free groups.

Fix a word-hyperbolic group $G$ and a finite generating set $T$. Let $C_T(G)$ denote the
Cayley graph of $G$ with respect to $T$. Let $\sigma$ be an oriented simplicial path in
$C_T(G)$. A {\em copy} of $\sigma$ is a translate $g\cdot \sigma$ for some $g \in G$.
If $\gamma$ is an oriented simplicial path in $C_T(G)$, let $|\gamma|_\sigma$
denote the maximal number of disjoint copies of $\sigma$ contained in $\gamma$.
For $g \in G$, define
$$c_\sigma(g) = d(\id,g) - \inf_\gamma(\length(\gamma) - |\gamma|_\sigma)$$
where the infimum is taken over all directed paths $\gamma$ in $C_T(G)$ from
$\id$ to $g$, and $d(\cdot,\cdot)$ denotes distance in $C_T(G)$.

\begin{definition}[Epstein-Fujiwara]
A {\em counting quasimorphism} on $G$ is a function of the form
$$h_\sigma(g):=c_\sigma(g) - c_{\sigma^{-1}}(g)$$
where $\sigma^{-1}$ denotes the same simplicial path as $\sigma$ with the opposite
orientation.
\end{definition}

Since $|\gamma|_\sigma$ 
takes discrete values, the infimum is realized in the definition of $c_\sigma$.
A path $\gamma$ for which 
$$c_\sigma(g) = d(\id,g) - \length(\gamma) + |\gamma|_\sigma$$ 
is called a {\em realizing path} for $g$. Realizing paths exist, and satisfy the
following geometric property:

\begin{lemma}[Epstein-Fujiwara, \cite{Epstein_Fujiwara} Prop.~2.2]\label{uniform_quasigeodesic_lemma}
Any realizing path for $g$ is a $K,\epsilon$-quasigeodesic in $C_T(G)$, where
$$K = \frac {\length(\sigma)} {\length(\sigma) -1}, \quad \text{ and } \quad 
\epsilon = \frac {2 \cdot \length(\sigma)} {\length(\sigma) - 1}$$
\end{lemma}

Moreover,

\begin{lemma}[Epstein-Fujiwara, \cite{Epstein_Fujiwara} Prop.~2.13]\label{uniform_defect}
Let $\sigma$ be a path in $C_T(G)$ of length at least $2$. Then there is a constant
$K(\delta)$ (where $T$ is such that $C_T(G)$ is $\delta$-hyperbolic as a metric space)
such that $D(h_\sigma)\le K(\delta)$.
\end{lemma}

Counting quasimorphisms are very versatile, as the following lemma shows:
\begin{lemma}\label{separation_lemma}
Let $G$ be a torsion-free, nonelementary word-hyperbolic group. 
Let $g_i$ be a finite collection of elements of $G$.
There is a commutator $s \in G'$ and a quasimorphism $\phi$ on $G$ with the following
properties:
\begin{enumerate}
\item{$|\phi(g_i)|=0$ for all $i$}
\item{$|\phi(s^n) -n| \le K_1$ for all $n$, where $K_1$ is a constant which depends only
on $G$.}
\item{$D(\phi) \le K_2$ where $K_2$ is a constant which depends only on $G$}
\end{enumerate}
\end{lemma}
\begin{proof}
Fix a finite generating set $T$ so that $C_T(G)$ is $\delta$-hyperbolic. There
is a constant $N$ such that for any nonzero $g \in G$, the power $g^N$ fixes an
axis $L_g$ (\cite{Gromov_hyperbolic}). Since $G$ is nonelementary, it contains quasigeodesically
embedded copies of free groups, of any fixed rank. So we can find a commutator $s$
whose translation length (in $C_T(G)$) is as big as desired. In particular,
given $g_1,\cdots,g_j$ we choose $s$ with $\tau(s) \gg \tau(g_i)$ for all $i$.
Let $L$ be a geodesic axis for $s^N$, and let $\sigma$ be a fundamental domain
for the action of $s^N$ on $L$. Since $|\sigma| = N\tau(s) \gg \tau(g_i)$, 
Lemma~\ref{uniform_quasigeodesic_lemma}
implies that there are no copies of $\sigma$ or $\sigma^{-1}$ in a realizing path
for any $g_i$. Hence $h_\sigma(g_i)=0$ for all $i$. By Lemma~\ref{uniform_defect},
$D(h_\sigma) \le K(\delta)$. It remains to estimate $h_\sigma(s^n)$.

In fact, the argument of \cite{Calegari_Fujiwara} Theorem~$\text{A}'$ 
(which establishes explicitly an estimate that is 
implicit in \cite{Epstein_Fujiwara}) shows that
for $N$ sufficiently large (depending only on $G$ and not on $s$) no copies
of $\sigma^{-1}$ are contained in any realizing path for $s^n$ with $n$ positive,
and therefore $|h_\sigma(s^n) - \lfloor n/N \rfloor|$ is bounded by a constant
depending only on $G$. The quasimorphism $\phi = N\cdot h_\sigma$ has the desired properties.
\end{proof}

\begin{remark}\label{torsion_remark}
The hypothesis that $G$ is torsion-free is included only to ensure that $s$ is not
conjugate to $s^{-1}$. It is possible to remove this hypothesis by taking slightly
more care in the definition of $s$, using the methods of the proof of Proposition~2 from
\cite{Bestvina_Fujiwara}. We are grateful to the referee for pointing this out.
\end{remark}

We now give the proof of Theorem~B:
\begin{proof}
Let $g,h \in G'$ have commutator length at least $R$. Let 
$g = s_1s_2 \cdots s_n$ and $h = t_1t_2 \cdots t_m$ where $n,m \ge R$ are equal
to the commutator lengths of $g$ and $h$ respectively, and
each $s_i,t_i$ is a commutator in $G$. Let $s$ be a commutator with the properties
described in Lemma~\ref{separation_lemma}
with respect to the elements $g,h$; that is, we want $s$ for which there is a
quasimorphism $\phi$ with $\phi(g)=\phi(h)=0$, with $|\phi(s^n)-n| \le K_1$ for all $n$,
and with $D(\phi) \le K_2$. Let $N \gg R$ be very large. We build a path
in $C_S(G')$ from $g$ to $h$ out of four segments, none of which come too close to
$\id$. 

The first segment is
$$g, gs, gs^2, gs^3, \cdots , gs^N$$
Since $s$ is a commutator, $d(gs^i,\id) \ge R-i$ for any $i$. On the other hand,
$$\phi(gs^i) \ge \phi(g) + \phi(s^i) - D(\phi) \ge i - K_2 - K_1$$
where $K_1,K_2$ are as in Lemma~\ref{separation_lemma} (and do not depend on $g,h,s$).
From Lemma~\ref{7_lipschitz} we can estimate
$$d(gs^i,\id) \ge \frac {\phi(gs^i)} {7D(\phi)} \ge \frac {i - K_2 - K_1} {7K_2}$$
Hence $d(gs^i,\id) \ge R/14K_2 - (K_1+K_2)/7K_2$ for all $i$, so providing $R \gg K_1,K_2$,
the path $gs^i$ never gets too close to $\id$.

The second segment is
$$gs^N = s_1s_2 \cdots s_ns^N, s_2\cdots s_ns^N,\cdots,s^N$$
Note that consecutive elements in this segment are distance $1$ apart in $C_S(G')$,
by Lemma~\ref{left_right_invariant}.
Since $d(gs^N,\id) \ge (N - K_2 - K_1)/7K_2 \gg R$ for $N$ sufficiently large,
we have $$d(s_i\cdots s_ns^N,\id) \gg R$$ for all $i$.

The third segment is
$$s^N, t_ms^N, t_{m-1}t_ms^N,\cdots, t_1t_2\cdots t_ms^N = hs^N$$
and the fourth is
$$hs^N,hs^{N-1},\cdots,hs,h$$
For the same reason as above, neither of these segments gets too close to $\id$.
This completes the proof of the theorem, taking $r = R/14K_2 - (K_1+K_2)/7K_2$.
\end{proof}

\section{Asymptotic dimension}\label{asymptotic_section}

The main point of this section is to make the observation that $G'$ for $G$ as
above is not a quasi-tree, and to restate this observation in terms of asymptotic dimension.
We think it is worth making this restatement explicitly.
The notion of asymptotic dimension is introduced in \cite{Gromov_asymptotic}, p.~32.

\begin{definition}
Let $X$ be a metric space, and $X = \cup_i U_i$ a covering by subsets. For given
$D\ge 0$, the {\em $D$-multiplicity} of the covering is at most $n$ if for any
$x \in X$, the closed $D$-ball centered at $x$ intersects at most $n$ of the $U_i$.

A metric space $X$ has {\em asymptotic dimension at most $n$} if for every $D \ge 0$
there is a covering $X = \cup_i U_i$ for which the diameters of the $U_i$ are uniformly
bounded, and the $D$-multiplicity of the covering is at most $n+1$. The least such $n$
is the {\em asymptotic dimension} of $X$, and we write
$$\asdim(X)=n$$
\end{definition}

If $X$ is a metric space, we say $H_1(X)$ is {\em uniformly generated} if there is
a constant $L$ such that $H_1(X)$ is generated by loops of length at most $L$.
It is clear that if $X$ is large scale $1$-connected, then $H_1(X)$ is uniformly
generated. Fujiwara-Whyte \cite{Fujiwara_Whyte} prove the following theorem:

\begin{theorem}[Fujiwara-Whyte, \cite{Fujiwara_Whyte}, Thm.~0.1]
Let $X$ be a geodesic metric space with $H_1(X)$ uniformly generated. $X$
has $\asdim(X)=1$ if and only if $X$ is quasi-isometric to an unbounded tree.
\end{theorem}

A group whose Cayley graph is quasi-isometric to an unbounded tree has more than
one end (see e.g. Manning \cite{Manning}, especially \S~2.1 and \S~2.2). 
Hence Theorem~A and Theorem~B together imply the following:

\begin{corollary}
Let $G$ be a nonelementary torsion-free word-hyperbolic group. Then
$$\asdim(C_S(G')) \ge 2$$
\end{corollary}

\section{Acknowledgment}
We would like to thank Koji Fujiwara for some useful conversations.
We would also like to thank the anonymous referee for a careful reading,
and many useful comments. Danny Calegari was partially funded by NSF grant DMS 0707130.

\end{document}